\documentclass[12pt]{article}
\usepackage{amsmath,amssymb}
\usepackage{amsthm}

\def\N{\mathbf{N}}

\def\1{\mathbf{1}}

\def\be{\beta}

\def\ep{\epsilon}
\def\de{\delta}

\newtheorem{prop}{Proposition}[section]
\newtheorem{theorem}{Theorem}[section]

\newtheorem{remark}{Remark}

\newcommand{\la}{\lambda}

\begin{document}
\title{Corruption and botnet defense: a mean field game approach
%\thanks{http://arxiv.org/abs/1507.03240}
\thanks{Supported by RFFI grant  No 14-06-00326}}
\author{
V. N. Kolokoltsov\thanks{Department of Statistics, University of Warwick,
 Coventry CV4 7AL UK,  Email: v.kolokoltsov@warwick.ac.uk and associate member  of Institute of Informatics Problems, FRC CSC RAS}
 and O. A. Malafeyev\thanks{Fac. of Appl. Math. and Control Processes, St.-Petersburg State Univ., Russia}}

\maketitle

\begin{abstract}
Recently developed toy models for the mean-field games of corruption and botnet defence in cyber-security
with three or four states of agents are extended to a more general mean-field-game model with $2d$ states, $d\in \N$.
 In order to tackle new technical difficulties arising from a larger state-space
we introduce new asymptotic regimes, namely small discount and small interaction asymptotics.
Moreover, the link between stationary and time-dependent solutions is established rigorously leading
to a performance of the turnpike theory in a mean-field-game setting.
\end{abstract}



\section{Introduction}

Toy models for the mean-field games of corruption and botnet defense in cyber-security were developed in
\cite{KolMa15}  and \cite{KolBen}. These were games with three and four states of the agents respectively.
Here we develop a more general mean-field-game model with $2d$ states, $d\in \N$, that extend the models of
\cite{KolMa15}  and \cite{KolBen}. In order to tackle new technical difficulties arising from a larger state-space
we introduce new asymptotic regimes, small discount and small interaction asymptotics. Hence the properties
 that we obtain for the new model do not cover more precise
results of \cite{KolMa15}  and \cite{KolBen} (with the full classification of the bifurcation points),
but capture their main qualitative and quantitative features and provide regular solutions away from the points of bifurcations.
Apart from new modeling, this paper contributes to one of the key questions in the modern study of mean-field games, namely, what
is the precise link between stationary and time -dependent solutions. This problem is sorted out here for a concrete model, but
the method can be definitely used in more general situations.

On the one hand, our model is a performance of the general pressure-and-resistance-game framework of \cite{KolPresRes}
and the nonlinear Markov battles of \cite{Ko12}, and on the other hand, it represents a simple example of mean-field- and
evolutionary-game modeling of networks. Initiating the development of the
latter, we stress already here that two-dimensional arrays of states arise naturally in many situations, one of the dimensions being
controlled mostly by the decision of agents (say, the level of tax evasion in the context of inspection games) and the other
one by a principal (major player) or evolutionary interactions (say, the level of agents in bureaucratic staircase,
the type of a computer virus used by botnet herder, etc).

We shall dwell upon two  basic interpretations of our model: corrupted bureaucrats playing against the principal (say, governmental
representative, also referred in literature as benevolent dictator) or computer owners playing against a botnet herder
(which then takes the role of the principal), which tries to infect the computers with viruses. Other interpretations can be
done, for instance,  in the framework of the inspection games (inspector and tax payers) or of the disease spreading
in epidemiology (among animals or humans), or the defense against a biological weapon.
Here we shall keep the principal in the background concentrating on the behavior of
small players (corrupted  bureaucrats or computer owners), which we shall refer to as agents or players.

The paper is organized as follows. In the next section we introduce our model
specifying in this context the basic notions of the mean-field-game (MFG) consistency problem in its dynamic and stationary versions.
In section \ref{secstationary}  we calculate explicitly all non-degenerate solutions of the stationary MFG problem.
In section \ref{sectimedepsol} we show how from a stationary solution one can construct a class of full time-dependent solutions
satisfying the so-called turnpike property around the stationary one.

We complete this introductory section with short bibliographical notes.

Analysis of the spread of corruption in bureaucracy is a well recognized area of the application of game theory,
which attracted attention of many researchers. General surveys can be found in \cite{Aidt}, \cite{Jain}, \cite{LeTsi98}.
More recent literature is reviewed in \cite{KolMa15}, see also \cite{Mala14}, \cite{Mala15} for
electric and engineering interpretations of corruption games.

The use of game theory in modeling attacker-defender has
been extensively adopted in the computer security domain recently, see \cite{BenKaHo},
\cite {LiLiSt} and \cite{LyWi} and bibliography
there for more details.

Mean-field games present a quickly developing area of the game theory.
Their study was initiated by Lasry-Lions \cite{LL2006} and Huang-Malhame-Caines
\cite{HCM3} and has been quickly developing since then,
see \cite{Ba13}, \cite{BenFr}, \cite{GLL2010}, \cite{Gomsurv}, \cite{Cain14} for recent surveys,
as well as \cite{Car13}, \cite{CarD13}, \cite{BasRa14}, \cite{KolTrYang14}
\cite{TemBas14}, \cite{BaTemBa} and references therein for some further developments.

\section{The model}
\label{secmodel}

We assume that any agent has $2d$ states: $iI$ and $iS$, where $i\in \{1, \cdots , d\}$ and is referred to as a strategy.
In the first interpretation the letters $S$ or $I$ designate the senior or initial position of a bureaucrat in the
hierarchical staircase and $i$ designates the level or type of corruptive behavior (say, the level of bribes one asks from customers
or, more generally, the level of illegal profit he/she aims at). In the literature on corruption the state $I$ is often denoted by $R$
and is referred to as the reserved state. It is interpreted as a job of the lowest salary given to the not trust-worthy bureaucrats.
In the second interpretation  the letters $S$ or $I$ designate
susceptible or infected states of computers and $i$ denotes the level or the type of defense system available on the market.

We assume that the choice of a strategy depends exclusively on the decision of an agent.
 The control parameter $u$ of each player may have $d$ values denoting the strategy the agent prefers at a moment.
 As long as this coincides with the current strategy, the updating of a strategy does not occur.
Once the decision to change $i$ to $j$ is made, the actual updating is supposed to occur with a certain rate $\la$.
Following \cite{KolBen}, we shall be mostly interested in the asymptotic regime of fast execution of individual decisions, that is, $\la \to \infty$.

 The change between $S$ and $I$ may have two causes: the action of the principal (pressure game component) and of the peers (evolutionary component).
In the first interpretation the principal can promote the bureaucrats from the initial to the senior position or degrade them
to the reserved initial position, whenever their illegal behavior is discovered. The peers can also take part in this process contributing to
the degrading of corrupted bureaucrats, for instance, when they trespass certain social norms.
In the second interpretation the principal, the botnet herder, infects computers with the virus by direct attacks turning $S$ to $I$,
and the virus then spreads through the network of computers by a pairwise interaction. The recovery change from $I$ to $S$
is due to some system of repairs which can be different in different protection levels $i$.

Let $q^i_+$ denote the recovery rates of upgrading from $iR$ to $iS$
and $q^i_-$ the rates of degrading (punishment or infection) from state $iR$ to $iS$,
which are independent of the state of other agents (pressure component), and let $\be_{ij}$ denote the rates
at which an agent in state $iI$ can stimulate the degrading (punishment or infection) of another agent from $jS$ to $jI$
(evolutionary component). For simplicity we ignore here the possibility of upgrading changes from $jS$ to $jI$ due to the interaction with peers.

A state of the system is a vector $n=(n_{1S}, n_{1I}, \cdots , n_{dS}, n_{dI})$
with coordinates presenting the number of agents in the corresponding states,
or its normalized version $x=(x_{1S}, x_{1I}, \cdots , x_{dS}, x_{dI})=n/N$ with $N=n_{1S}+ n_{1I}+ \cdots + n_{dS}+ n_{dI}$
the total number of agents.

Therefore, assuming that all players have the same strategy $u^{com}_t=\{u^{com}(iS), u^{com}(iI)\}$,
 the evolution of states in the limit of large number of players $N\to \infty$ is given by the equations
 \begin{equation}
 \label{eqmainkineticbotnetmultiev}
\begin{aligned}
& \dot x_{iI} =\la \sum_{j\neq i} x_{jI} \1(u^{com}(jI)=i)-\la \sum_{j\neq i} x_{iI} \1(u^{com}(iI)=j)
 + x_{iS} q_-^i  -x_{iI} q_+^i +\sum_j x_{iS}x_{jI} \be_{ji}, \\
& \dot x_{iS} =\la \sum_{j\neq i} x_{jS} \1(u^{com}(jS)=i)- \la \sum_{j\neq i} x_{iS} \1(u^{com}(iS)=j)
 - x_{iS} q_-^i +x_{iI} q_+^i -\sum_j x_{iS}x_{jI} \be_{ji},
 \end{aligned}
 \end{equation}
 for all $i=1, \cdots ,d$. Here and below $\1(M)$ denotes the indicator function of a set $M$.

\begin{remark}
It is well known that evolutions of this type can be derived rigorously as the dynamic law of large numbers
for the corresponding Markov models of a finite number of players, see detail e.g. in \cite{Ko12} or \cite{KolPresRes}.
\end{remark}

To specify the optimal behavior of agents we have to introduce payoffs in different states and possibly
 costs for transitions. For simplicity we shall ignore here the latter.
Talking about corrupted agents it is natural to talk about maximizing profit, while
talking about infected computers it is natural to talk about minimizing costs.
To unify the exposition we shall deal with the minimization of costs, which is
equivalent to the maximization of their negations.

Let $w_I^i$ and $w_S^i$ denote the costs per time-unit of staying in $iI$ and $iS$ respectively.
According to our interpretation of $S$ as a better state, $w^i_S<w^i_I$ for all $i$.

Given the evolution of the states $x=x(s)$ of the whole system on a time interval $[t,T]$, the individually optimal costs $g(iI)$ and $g(iS)$ and
individually optimal control $u^{ind}_s(iI)$ and $u^{ind}_s(iS)$ can be found from the HJB equation

\begin{equation}
 \label{eqHJBbotnetmultiev}
\begin{aligned}
& \dot g(iI)+\la \min_u \sum_{j=1}^d \1(u(iI)=j)(g(jI)-g(iI)) +q^i_+(g(iS)-g(iI)) +w_I^i=0, \\
& \dot g(iS)+\la \min_u \sum_{j=1}^d \1(u(iS)=j)(g(jS)-g(iS)) +q^i_- (g(iI)-g(iS)) \\
& \quad \quad \quad + \sum_{j=1}^d \be_{ji} x_{jI}(s)(g(iI)-g(iS))+w_S^i=0.
\end{aligned}
\end{equation}

The basic {\it MFG consistency equation} for a time interval $[t,T]$ can now be written as $u_s^{com}=u_s^{ind}$.

\begin{remark}
The reasonability of this condition in the setting of the large number of players is more or less obvious.
And in fact in many situations it was proved rigorously that
its solutions represent the $\ep$-Nash equilibrium for the corresponding Markov model of $N$ players,
with $\ep \to 0$ as $N\to \infty$, see e.g. \cite{BasRa14} for finite state models considered here.
\end{remark}

In this paper we shall mostly work with discounted payoff with the discounting coefficient $\de>0$, in which case the HJB equation
for the discounted optimal payoff $e^{-s\de}g$ of an individual player with any time horizon $T$ writes down as

\begin{equation}
 \label{eqHJBbotnetmultievdisc}
\begin{aligned}
& \dot g(iI)+\la \min_u \sum_{j=1}^d \1(u(iI)=j)(g(jI)-g(iI)) +q^i_+(g(iS)-g(iI)) +w_I^i=\de g(iI), \\
& \dot g(iS)+\la \min_u \sum_{j=1}^d \1(u(iS)=j)(g(jS)-g(iS)) +q^i_- (g(iI)-g(iS)) \\
& \quad \quad \quad + \sum_{j=1}^d \be_{ji} x_{jI}(s)(g(iI)-g(iS))+w_S^i=\de g(iS).
\end{aligned}
\end{equation}

Notice that since this is an equation in a Euclidean space with Lipschitz coefficients, it has a unique solution
for $s\le T$ and any given boundary condition $g$ at time $T$ and any measurable functions $x_{iI}(s)$.

For the discounted payoff the basic {\it MFG consistency equation} $u_s^{com}=u_s^{ind}$
for a time interval $[t,T]$ can be reformulated by saying that $x,u,g$ solve
the coupled forward-backward system \eqref{eqmainkineticbotnetmultiev}, \eqref{eqHJBbotnetmultievdisc}, so that $u_s^{com}$
used in \eqref{eqmainkineticbotnetmultiev} coincide with the minimizers in \eqref{eqHJBbotnetmultievdisc}.
The main objective of the paper is to provide a general class of solutions of the discounted MFG consistency equation
with stationary (time-independent) controls $u_s^{com}$.

As a first step to this objective we shall analyse the fully stationary solutions, when the evolution
 \eqref{eqmainkineticbotnetmultiev} is replaced by the corresponding fixed point condition:

 \begin{equation}
 \label{eqmainkineticbotnetmultist}
\begin{aligned}
& \la \sum_{j\neq i} x_{jI} \1(u^{com}(jI)=i)-\la \sum_{j\neq i}  x_{iI} \1(u^{com}(iI)=j)
 +x_{iS} q_-^i  -x_{iI} q_+^i +\sum_j x_{iS}x_{jI} \be_{ji}=0, \\
& \la \sum_{j\neq i} x_{jS} \1(u^{com}(jS)=i)-\la \sum_{j\neq i}  x_{iS} \1(u^{com}(iS)=j)
 -x_{iS} q_-^i +x_{iI} q_+^i -\sum_j x_{iS}x_{jI} \be_{ji}=0.
 \end{aligned}
 \end{equation}

There are two standard stationary optimization problems naturally linked with a dynamic one,
one being the search for the average payoff for long period game, and another the search for
discounted optimal payoff. The first is governed by the solutions of HJB of the form $(T-s)\mu +g$,
linear in $s$ (then $\mu$ describing the optimal average payoff),
so that $g$ satisfies the stationary HJB equation:

\begin{equation}
 \label{eqHJBbotnetmultiav}
\begin{aligned}
& \la \min_u \sum_{j=1}^d \1(u(iI)=j)(g(jI)-g(iI)+c_{ij}) +q^i_+(g(iS)-g(iI)) +w_I^i =\mu, \\
& \la \min_u \sum_{j=1}^d \1(u(iS)=j)(g(jS)-g(iS)+c_{ij}) +q^i_- (g(iI)-g(iS)) \\
& \quad \quad \quad + \sum_{j=1}^d \be_{ji} x_{jI}(g(iI)-g(iS))+w_S^i =\mu.
\end{aligned}
\end{equation}

In the second problem, if the discounting coefficient is $\de$,
the stationary discounted optimal  payoff $g$  satisfies the stationary version of \eqref{eqHJBbotnetmultievdisc}:

\begin{equation}
 \label{eqHJBbotnetmultidis}
\begin{aligned}
& \la \min_u \sum_{j=1}^d \1(u(iI)=j)(g(jI)-g(iI)+c_{ij}) +q^i_+(g(iS)-g(iI)) +w_I^i =\de g(iI), \\
& \la \min_u \sum_{j=1}^d \1(u(iS)=j)(g(jS)-g(iS)+c_{ij}) +q^i_- (g(iI)-g(iS)) \\
& \quad \quad \quad + \sum_{j=1}^d \be_{ji} x_{jI}(g(iI)-g(iS))+w_S^i =\de g(iS).
\end{aligned}
\end{equation}

In \cite{KolMa15}  and \cite{KolBen} we concentrated on the first approach, and here
we shall concentrate on the second one, with a discounted payoff.
The {\it stationary MFG consistency condition}
is the coupled system of equations \eqref{eqmainkineticbotnetmultist} and
\eqref{eqHJBbotnetmultidis}, so that the individually optimal stationary control
$u^{ind}$ found from \eqref{eqHJBbotnetmultidis} coincides with
the common stationary control $u^{com}$ from \eqref{eqmainkineticbotnetmultist}.

For simplicity we shall be interested in {\it non-degenerate controls} $u^{ind}$
characterized by the condition that the minimum in \eqref{eqHJBbotnetmultidis}
is always attained on a single value of $u$.

A new technical novelty as compared with \cite{KolBen} and \cite{KolMa15}
will be systematic working in the asymptotic regimes of small discount $\de$
and small interaction coefficients $\be_{ij}$. This approach leads to more or
less explicit calculations of stationary MFG solutions and their further justification.

\section{Stationary MFG problem}
\label{secstationary}

The following result identifies all possible stationary non-degenerate controls that can occur as
solutions of  \eqref{eqHJBbotnetmultidis}.

\begin{prop}
\label{proponstrate}
Non-degenerate controls solving \eqref{eqHJBbotnetmultidis} could be only of the type
$[i(I),k(S)]$: switch to strategy $i$ when in $I$ and to $k$ when in $S$.
\end{prop}

\begin{proof}
If moving from the strategy $k$ to the strategy $i$ is optimal, then $g(i) < g(l)$ for all $l$ and hence moving from $m$ to $i$
is optimal for any $m$.
\end{proof}

Let us consider first the control $[i(I),i(S)]$ denoting it by $\hat u^i$:
\[
\hat u^i(jS)=\hat u^i(jI)=i, \quad j=1, \cdots, d.
\]
We shall refer to the control $\hat u^i$ as the one with the strategy $i$ individually optimal.

The control $\hat u^i$ and the corresponding distribution $x$ solve
the stationary MFG problem if
they solve the corresponding HJB \eqref{eqHJBbotnetmultidis}, that is
\begin{equation}
 \label{eqHJBOpsinglei1}
\left\{
\begin{aligned}
& q^i_+(g(iS)-g(iI)) +w^i_I =\de g(iI), \\
& q^i_-(g(iI)-g(iS)) + \sum_{k}\be_{ki} x_{kI}(g(iI)-g(iS))+w_S^i =\de g(iS), \\
& \la (g(iI)-g(jI)) +q^j_+(g(jS)-g(jI)) +w_I^j =\de g(jI), \quad j\neq i, \\
& \la (g(iS)-g(jS)) +q^j_- (g(jI)-g(jS)) \\
& \quad \quad \quad + \sum_k\be_{kj} x_{kI}(g(jI)-g(jS))+w_S^j =\de g(jS), \quad j\neq i,
\end{aligned}
\right.
\end{equation}
where for all $j\neq i$
\begin{equation}
 \label{eqHJBOpsinglei2}
g(iI) \le g(jI), \quad g(iS) \le g(jS),
\end{equation}
and $x$ is a fixed point of the evolution \eqref{eqmainkineticbotnetmultist}  with $u^{com}=\hat u^i$, that is

\begin{equation}
 \label{eqkineticstsinglei1}
\left\{
\begin{aligned}
& x_{iS} q_-^i -x_{iI} q_+^i +\sum_j x_{iS}x_{jI} \be_{ji}
 +\la \sum_{j\neq i} x_{jI}=0, \\
&  -x_{iS} q_-^i +x_{iI} q_+^i -\sum_j x_{iS}x_{jI} \be_{ji}
 +\la \sum_{j\neq i} x_{jS}=0, \\
 & x_{jS} q_-^j -x_{jI} q_+^j +\sum_k x_{jS}x_{kI} \be_{kj}
 -\la x_{jI}=0, \quad j\neq i,\\
& -x_{jS} q_-^j +x_{jI} q_+^j -\sum_k x_{jS}x_{kI} \be_{ji}
 -\la x_{jS}=0, \quad j\neq i.
 \end{aligned}
\right.
 \end{equation}

 This solution $(\hat u^i,x)$ is stable if $x$ is a stable fixed point of the evolution \eqref{eqmainkineticbotnetmultiev}
 with $u^{com}=\hat u^i$, that is,  of the evolution

 \begin{equation}
 \label{eqkineticstsinglei2}
\left\{
\begin{aligned}
& \dot x_{iI}=x_{iS} q_-^i -x_{iI} q_+^i +\sum_j x_{iS}x_{jI} \be_{ji}
 +\la \sum_{j\neq i} x_{jI}, \\
&  \dot x_{iS} =-x_{iS} q_-^i +x_{iI} q_+^i -\sum_j x_{iS}x_{jI} \be_{ji}
 +\la \sum_{j\neq i} x_{jS}, \\
 & \dot x_{jI} = x_{jS} q_-^j -x_{jI} q_+^j +\sum_k x_{jS}x_{kI} \be_{kj}
 -\la x_{jI}, \quad j\neq i,\\
& \dot x_{jS} = -x_{jS} q_-^j +x_{jI} q_+^j -\sum_k x_{jS}x_{kI} \be_{ji}
 -\la x_{jS}, \quad j\neq i.
 \end{aligned}
\right.
 \end{equation}

Adding together the last two equations of \eqref{eqkineticstsinglei1} we find that $x_{jI}=x_{jS}=0$ for $j\neq i$,
as one could expect. Consequently, the whole system \eqref{eqkineticstsinglei1} reduces to the single equation
\[
x_{iS}q^i_- +x_{iI} \be_{ii} x_{iS} -x_{iI} q^i_+ =0,
\]
which, for $y=x_{iI}$, $1-y=x_{iS}$, yields the quadratic equation
\[
Q(y)= \be_{ii}y^2+y(q^i_+ -\be_{ii} +q^i_-) -q^i_-=0,
\]
with the unique solution on the interval $(0,1)$:
\begin{equation}
\label{eqkineticstsinglei3}
x^*=\frac{1}{2\be_{ii}}\left[ \be_{ii}-q^i_+-q^i_-
+\sqrt{(\be_{ii} +q^i_-)^2+(q^i_+)^2-2 q^i_+ (\be_{ii} -q^i_-)}\right].
\end{equation}

To analyze stability of the fixed point $x_{iI}=x^*, x_{iS}=1-x^*$ and $x_{jI}=x_{jS}=0$ for $j\neq i$,
we introduce the variables $y=x_{iI}-x^*$. In terms of $y$ and $x_{jI},x_{jS}$ with $j\neq i$, system
\eqref{eqkineticstsinglei2} rewrites as
\begin{equation}
 \label{eqkineticstsinglei4}
\left\{
\begin{aligned}
& \dot y =[1-x^*-y-\sum_{j\neq i} (x_{jI}+x_{jS})] [q_- +\sum_{k\neq i} x_{kI} \be_{ki}+(y+x^*)\be_{ii}]
 -(y+x^*) q^i_+ +\la \sum_{j\neq i} x_{jI}, \\
 & \dot x_{jI} =x_{jS} [q_-^j  +\sum_{k\neq i}x_{kI} \be_{kj}+(y+x^*) \be_{ij}]
 -x_{jI} q_+^j -\la x_{jI}, \quad j\neq i,\\
& \dot x_{jS} =-x_{jS} [q_-^j +\sum_{k\neq i}x_{kI} \be_{kj}+(y+x^*) \be_{ij}]
 +x_{jI} q_+^j -\la x_{jS}, \quad j\neq i.
 \end{aligned}
\right.
\end{equation}

Its linearized version around the fixed point zero is
\[
\left\{
\begin{aligned}
& \dot y =(1-x^*)(\sum_{k\neq i} x_{kI} \be_{ki}+y\be_{ii})-[y+\sum_{k\neq i} (x_{kI}+x_{kS})](q_-^i +x^*\be_{ii})
 -y q^i_+ +\sum_{k\neq i} \la x_{kI}, \\
 & \dot x_{jI} =x_{jS} (q_-^j +x^* \be_{ij})
 -x_{jI} q_+^j -\la x_{jI}, \quad j\neq i,\\
& \dot x_{jS} =-x_{jS} (q_-^j +x^* \be_{ij})
 +x_{jI} q_+^j -\la x_{jS}, \quad j\neq i.
 \end{aligned}
\right.
\]

Since the equations for $x_{jI}, x_{jS}$ contain neither $y$ nor other variables, the eigenvalues
of this linear system are
\[
\xi_i=(1-2x^*)\be_{ii} -q^i_- -q^i_+,
\]
and $(d-1)$ pairs of eigenvalues arising from $(d-1)$ systems
\[
\left\{
\begin{aligned}
& \dot x_{jI} =x_{jS} (q_-^j +x^* \be_{ij})
 -x_{jI} q_+^j -\la x_{jI}, \quad j\neq i,\\
& \dot x_{jS} =-x_{jS} (q_-^j +x^* \be_{ij})
 +x_{jI} q_+^j -\la x_{jS}, \quad j\neq i,
 \end{aligned}
\right.
\]
that is

\[
\left\{
\begin{aligned}
& \xi_1^j = -\la -(q_+^j+q^j_-+x^* \be_{ii}) \\
& \xi_2^j = -\la.
\end{aligned}
\right.
\]

These eigenvalues being always negative, the condition of stability is reduced to the negativity of the first eigenvalue $\xi_i$:
\[
2x^*> 1-\frac{q^i_+ +q^i_-}{\be_{ii}}.
\]
But this is true due to \eqref{eqkineticstsinglei3} implying that this fixed point is always stable (by the Grobman-Hartman theorem).

Next, the HJB equation \eqref{eqHJBOpsinglei1} takes the form
\begin{equation}
 \label{eqHJBOpsinglei3}
\left\{
\begin{aligned}
& q^i_+(g(iS)-g(iI)) +w^i_I =\de g(iI), \\
& q^i_-(g(iI)-g(iS)) + \be_{ii} x^*(g(iI)-g(iS))+w_S^i =\de g(iS), \\
& \la (g(iI)-g(jI)) +q^j_+(g(jS)-g(jI)) +w_I^j =\de g(jI), \quad j\neq i, \\
& \la (g(iS)-g(jS)) +q^j_- (g(jI)-g(jS)) \\
& \quad \quad \quad + \be_{ij} x^*(g(jI)-g(jS))+w_S^j =\de g(jS), \quad j\neq i,
\end{aligned}
\right.
\end{equation}

Subtracting the first equation from the second one yields
\begin{equation}
 \label{eqHJBOpsinglei4}
g(iI)-g(iS)=\frac{w^i_I-w^i_S}{q_-^i+q_+^i+\be_{ii}x^* +\de}.
\end{equation}
In particular, $g(iI)>g(iS)$ always, as expected.
Next, by the first equation of \eqref{eqHJBOpsinglei3},
\begin{equation}
 \label{eqHJBOpsinglei5}
\de g(iI)=w^i_I- \frac{q^i_+(w^i_I-w^i_S)}{q_-^i+q_+^i+\be_{ii}x^* +\de}.
\end{equation}

Consequently,
\begin{equation}
 \label{eqHJBOpsinglei6}
\de g(iS)=w^i_I- \frac{(q^i_++\de)(w^i_I-w^i_S)}{q_-^i+q_+^i+\be_{ii}x^* +\de}
=w^i_S+ \frac{(q^i_-+\be_{ii}x^*)(w^i_I-w^i_S)}{q_-^i+q_+^i+\be_{ii}x^* +\de}.
\end{equation}

Subtracting the third equation of \eqref{eqHJBOpsinglei3} from the fourth one yields
\[
(\la +q^j_++q^j_- +\be_{ii}x^*+\de)(g(jI)-g(jS))-\la (g(iI)-g(iS))=w^i_I-w^i_S,
\]
implying
\[
g(jI)-g(jS)=\frac{w^j_I-w^j_S+\la (g(iI)-g(iS))}{\la +q^j_++q^j_-+\be_{ij}x^*+\de}
\]
\begin{equation}
 \label{eqHJBOpsinglei7}
=g(iI)-g(iS)+[(w^j_I-w^j_S)-(g(iI)-g(iS))(q^j_++q^j_-+\be_{ij}x^*+\de)]\la^{-1} +O(\la^{-2}).
\end{equation}

From the fourth equation of \eqref{eqHJBOpsinglei3} it now follows that
\[
(\de +\la) g(jI)=w^j_I-q^j_+(g(jI)-g(jS))+\la g(iI),
\]
so that
\begin{equation}
 \label{eqHJBOpsinglei8}
g(jI)=g(iI)+[w^j_I-q^j_+(g(iI)-g(iS)) -\de g(iI)]\la^{-1} +O(\la^{-2}).
\end{equation}

Consequently,
\[
g(jS)=g(jI)-(g(jI)-g(jS))
\]
\begin{equation}
 \label{eqHJBOpsinglei9}
=g(iS)+[w^j_S+(q^j_- +\be_{ii}x^*)(g(iI)-g(iS)) -\de g(iS)]\la^{-1} +O(\la^{-2}).
\end{equation}

Thus the consistency conditions \eqref{eqHJBOpsinglei2} in the main order in $\la \to \infty$ become
\[
w^j_I-q^j_+(g(iI)-g(iS)) -\de g(iI)\ge 0, \quad
w^j_S+(q^j_- +\be_{ii}x^*)(g(iI)-g(iS)) -\de g(iS) \ge 0,
\]
or equivalently
\begin{equation}
 \label{eqHJBOpsinglei10}
w^j_I-w^i_I  \ge \frac{(q^j_+ -q^i_+)(w^i_I-w^i_S)}{q_-^i+q_+^i+\be_{ii}x^* +\de}, \quad
w^j_S-w^i_S  \ge \frac{[q^i_- -q^j_- +(\be_{ii}-\be_{ij})x^*](w^i_I-w^i_S)}{q_-^i+q_+^i+\be_{ii}x^* +\de}.
\end{equation}

In the first order in small $\be_{ii}$ this gets the simpler form, independent of $x^*$:
\begin{equation}
 \label{eqHJBOpsinglei11}
\frac{w^j_I-w^i_I}{w^i_I-w^i_S}  \ge \frac{q^j_+ -q^i_+}{q_-^i+q_+^i +\de}, \quad
\frac{w^j_S-w^i_S}{w^i_I-w^i_S}  \ge \frac{q^i_- -q^j_-}{q_-^i+q_+^i +\de}.
\end{equation}

Summarizing, we proved the following.

\begin{prop}
\label{propstriiop}
If \eqref{eqHJBOpsinglei11} holds for all $j\neq i$ with the strict inequality, then for sufficiently
 large $\la$ and sufficiently small $\be_{ij}$ there exists a unique solution
to the stationary MFG consistency problem  \eqref{eqmainkineticbotnetmultist} and
\eqref{eqHJBbotnetmultidis}  with the optimal control $\hat u^i$, the stationary distribution is  $x_i^I=x^*, x_i^S=1-x^*$ with
$x^*$ given by \eqref{eqkineticstsinglei3} and it is stable; the optimal payoffs are given by
\eqref{eqHJBOpsinglei5}, \eqref{eqHJBOpsinglei6}, \eqref{eqHJBOpsinglei8}, \eqref{eqHJBOpsinglei9}.
Conversely, if for all sufficiently large $\la$ there exists a solution
to the stationary MFG consistency problem  \eqref{eqmainkineticbotnetmultist} and
\eqref{eqHJBbotnetmultidis}  with the optimal control $\hat u^i$, then
\eqref{eqHJBOpsinglei10} holds.
\end{prop}

Let us turn to control $[i(I),k(S)]$ with $k\neq i$ denoting it by $\hat u^{i,k}$:
\[
\hat u^{i,k}(jS)=k, \quad \hat u^{i,k}(jI)=i, \quad j=1, \cdots, d.
\]

%For small discounting (so effectively very long play) and fast decision execution it does not seem to be very productive to turn around several strategies.
%It is remarkable however (and not at all trivial) that we can prove that the controls $[i(I),k(S)]$ with $k\neq i$ cannot be part of a stationary
%solution to the discounted HJB, as the following result shows.

The fixed point condition under $u^{com}=\hat u^{i,k}$ takes the form

\begin{equation}
 \label{eqkineticstsingleik1}
\left\{
\begin{aligned}
& x_{iS} q_-^i  -x_{iI} q_+^i +\sum_j x_{iS}x_{jI} \be_{ji} +\la \sum_{j\neq i} x_{jI} =0 \\
& -x_{iS} q_-^i +x_{iI} q_+^i -\sum_j x_{iS}x_{jI} \be_{ji} -\la x_{iS}=0 \\
& x_{kS} q_-^k  -x_{kI} q_+^k +\sum_j x_{kS}x_{jI} \be_{jk} -\la x_{kI} =0 \\
& -x_{kS} q_-^i +x_{kI} q_+^k -\sum_j x_{kS}x_{jI} \be_{jk} +\la \sum_{j\neq k} x_{jS}=0 \\
& x_{lS} q_-^l  -x_{lI} q_+^l +\sum_j x_{lS}x_{jI} \be_{jl} -\la x_{lS} =0 \\
& -x_{lS} q_-^l +x_{lI} q_+^l -\sum_j x_{lS}x_{jI} \be_{jl} -\la x_{lI} =0,
 \end{aligned}
\right.
 \end{equation}
 where $l\neq i,k$.

Adding the last two equations yields $x_{lI}+x_{lS}=0$ and hence $x_{lI}=x_{lS}=0$ for all $l \neq i,k$, as one could expect.
Consequently, for indices $i,k$ the system gets the form

\begin{equation}
 \label{eqkineticstsingleik2}
\left\{
\begin{aligned}
& x_{iS} q_-^i  -x_{iI} q_+^i + x_{iS}x_{iI} \be_{ii} + x_{iS}x_{kI} \be_{ki} +\la x_{kI} =0 \\
& -x_{iS} q_-^i +x_{iI} q_+^i -x_{iS}x_{iI} \be_{ii} - x_{iS}x_{kI} \be_{ki} -\la x_{iS}=0 \\
& x_{kS} q_-^k  -x_{kI} q_+^k + x_{kS}x_{kI} \be_{kk} + x_{kS}x_{iI} \be_{ik} -\la x_{kI}  =0 \\
& -x_{kS} q_-^i +x_{kI} q_+^k - x_{kS}x_{kI} \be_{kk} - x_{kS}x_{iI} \be_{ik} +\la x_{iS} =0
 \end{aligned}
 \right.
 \end{equation}

Adding the first two equation (or the last two equations) yields $x_{kI}=x_{iS}$.
Since by normalization
\[
x_{kS}=1- x_{iS}-x_{kI}-x_{iI}= 1-x_{iI}-2x_{kI},
\]
we are left with two equations only:

\begin{equation}
 \label{eqkineticstsingleik3}
\left\{
\begin{aligned}
& x_{kI} q_-^i  -x_{iI} q_+^i + x_{kI}x_{iI} \be_{ii} + x^2_{kI} \be_{ki} +\la x_{kI} =0 \\
& (1-x_{iI}-2x_{kI}) (q_-^k  +x_{kI} \be_{kk} + x_{iI} \be_{ik}) -(\la +q^k_+)x_{kI}  =0.
 \end{aligned}
 \right.
 \end{equation}

 From the first equation we obtain
 \[
 x_{iI}=\frac{ \la x_{kI}+\be_{ki}x_{kI}^2+q^i_-x_{kI}}{q^i_+-x_{kI}\be_{ii}}=\frac{ \la x_{kI}}{q^i_+-x_{kI}\be_{ii}}(1+O(\la^{-1})).
 \]
 Hence $x_{kI}$ is of order $1/\la$, and therefore
 \begin{equation}
 \label{eqkineticstsingleik4}
x_{iI}=\frac{ \la x_{kI}}{q^i_+}(1+O(\la^{-1})) \Longleftrightarrow x_{kI}=\frac{x_{iI} q^i_+}{\la}(1+O(\la^{-1})).
 \end{equation}

 In the major order in large $\la$ asymptotics, the second equation of \eqref{eqkineticstsingleik3} yields
 \[
 (1-x_{iI})(q^k_- +\be_{ik} x_{iI})-q^i_+x_{iI}=0
 \]
 or  for $y=x_{iI}$
 \[
Q(y)= \be_{ik}y^2+y(q^i_+ -\be_{ik} +q^k_-) -q^k_-=0,
\]
which is effectively the same equation as the one that appeared in the analysis of the control $[i(I),i(S)]$.
It has the unique solution on the interval $(0,1)$:
\begin{equation}
\label{eqkineticstsingleik5}
x_{iI}^*=\frac{1}{2\be_{ik}}\left[ \be_{ik}-q^i_+-q^k_-
+\sqrt{(\be_{ik} +q^k_-)^2+(q^i_+)^2-2 q^i_+ (\be_{ik} -q^k_-)}\right].
\end{equation}

Let us note that for small $\be_{ik}$ it expands as
\begin{equation}
\label{eqkineticstsingleik6}
x_{iI}^*=\frac{q^k_-}{q^k_-+q^i_+}+O(\be)=\frac{q^k_-}{q^k_-+q^i_+}+\frac{q^k_- q^i_+ }{(q^k_-+q^i_+)^3}\be +O(\be^2).
\end{equation}

Similar (a bit more lengthy) calculations, as for the control $[i(I),i(S)]$ show that
the obtained fixed point of evolution \eqref{eqmainkineticbotnetmultiev} is always stable.
We omit the detail, as they are the same as given in \cite{KolBen} for the case $d=2$.

Let us turn to the HJB equation \eqref{eqHJBOpsinglei1}, which under control $[i(I),k(S)]$ takes the form
\begin{equation}
 \label{eqHJBOpsingleik3}
\left\{
\begin{aligned}
& q^i_+(g(iS)-g(iI)) +w^i_I =\de g(iI), \\
& \la (g(kS)-g(iS))+\tilde q^i_-(g(iI)-g(iS))+w_S^i =\de g(iS), \\
& \la (g(iI)-g(kI)) +q^k_+(g(kS)-g(kI)) +w_I^k =\de g(kI),  \\
& \tilde q^k_- (g(kI)-g(kS))+w^k_S=\de g(kS) \\
& \la (g(iI)-g(jI)) +q^j_+(g(jS)-g(jI)) +w_I^j =\de g(jI), \quad j\neq i,k, \\
& \la (g(kS)-g(jS)) +\tilde q^j_- (g(jI)-g(jS))+w_S^j =\de g(jS), \quad j\neq i,k,
\end{aligned}
\right.
\end{equation}
supplemented by the consistency condition
\begin{equation}
 \label{eqHJBOpsingleik2}
g(iI) \le g(jI), \quad g(kS) \le g(jS),
\end{equation}
for all $j$,
where we introduced the notation
\begin{equation}
 \label{eqHJBOpsingleik4}
\tilde q^j_- =\tilde q^j_-(i,k)= q^j_- + \be_{ij}x_{iI}+\be_{kj}x_{kI}.
\end{equation}

The first four equations do not depend on the rest of the system and can be solved independently.
To begin with, we use the first and the fourth equation to find
\begin{equation}
 \label{eqHJBOpsingleik5}
g(iS)= g(iI)+\frac{\de g(iI)-w^i_I}{q^i_+}, \quad g(kI)= g(kS)+\frac{\de g(kS)-w^k_S}{\tilde q^k_-}.
\end{equation}
Then the second and the third equations can be written as the system for the variables $g(kS)$ and $g(iI)$:
\[
\left\{
\begin{aligned}
& \la g(kS)-(\la +\de)g(iI) -(\la +\de+\tilde q^i_-)\frac{\de g(iI)-w^i_I}{q^i_+} +w_S^i =0 \\
& \la g(iI)- (\la +\de) g(kS)-(\la+\de +q^k_+)\frac{\de g(kS)-w^k_S}{\tilde q^k_-} +w_I^k =0,
\end{aligned}
\right.
\]
or simpler as
\begin{equation}
 \label{eqHJBOpsingleik6}
\left\{
\begin{aligned}
& \la q^i_+ g(kS)-[\la (q^i_+ +\de) +\de(q^i_++\tilde q^i_-+\de)]g(iI) = -w^i_I(\la +\de +\tilde q^i_-)-w_S^i q^i_+ \\
& [\la (\tilde q^k_- +\de) +\de (\tilde q^k_-+q^k_++\de)] g(kS) -\la \tilde q^k_- g(iI) =w^k_I \tilde q^k_-+w^k_S(\la+\de +q^k_+).
\end{aligned}
\right.
\end{equation}

Let us find the asymptotic behavior of the solution for large $\la$. To this end
let us write
\[
g(iS)=g^0(iS)+\frac{g^1(iS)}{\la} + O(\la^{-2})
\]
with similar notations for other values of $g$.
Dividing \eqref{eqHJBOpsingleik6} by $\la$ and preserving only the leading terms in $\la$
we get the system
\begin{equation}
 \label{eqHJBOpsingleik7}
\left\{
\begin{aligned}
& q^i_+ g^0(kS)-(q^i_+ +\de) g^0(iI) = -w^i_I, \\
& (\tilde q^k_- +\de) g^0(kS) -\tilde q^k_- g^0(iI) =w^k_S.
\end{aligned}
\right.
\end{equation}
Solving this system and using \eqref{eqHJBOpsingleik5} to find the corresponding leading terms $g^0(iS)$, $g^0(kI)$ yields

\begin{equation}
 \label{eqHJBOpsingleik8}
 \begin{aligned}
& g^0(iS)=g^0(kS)=\frac{1}{\de} \frac{\tilde q^k_-w^i_I+q^i_+w^k_S+\de w^k_S}{\tilde q^k_-+q^i_+ +\de}, \\
& g^0(kI)=g^0(iI)=\frac{1}{\de} \frac{\tilde q^k_-w^i_I+q^i_+w^k_S+\de w^i_I}{\tilde q^k_-+q^i_+ +\de}.
\end{aligned}
\end{equation}

The remarkable equations $g^0(iS)=g^0(kS)$ and $g^0(kI)=g^0(iI)$ arising from the calculations have natural interpretation:
for instantaneous execution of personal decisions the discrimination between strategies $i$ and $j$ is not possible.
Thus to get the conditions ensuring \eqref{eqHJBOpsingleik2} we have to look for the next order of expansion in $\la$.

Keeping in \eqref{eqHJBOpsingleik6} the terms of zero-order in $1/\la$ yields the system
\begin{equation}
 \label{eqHJBOpsingleik9}
\left\{
\begin{aligned}
& q^i_+ g^1(kS)-(q^i_+ +\de) g^1(iI) =\de(q^i_++\tilde q^i_-+\de)g^0(iI) -w^i_I(\de +\tilde q^i_-)-w_S^i q^i_+ \\
& (\tilde q^k_- +\de) g^1(kS) -\tilde q^k_- g^1(iI)= -\de (\tilde q^k_-+q^k_++\de) g^0(kS) +w^k_I \tilde q^k_- +w^k_S(\de +q^k_+).
\end{aligned}
\right.
\end{equation}
Taking into account \eqref{eqHJBOpsingleik8}, conditions $g(iI)\le g(kI)$ and $g(kS)\le g(iS)$ turn to
\begin{equation}
 \label{eqHJBOpsingleik10}
\tilde q^k_- g^1(iI)\le g^1(kS)(\tilde q^k_- + \de), \quad q^i_+ g^1(kS)\le g^1(iI) (q^i_+ +\de).
\end{equation}

Solving \eqref{eqHJBOpsingleik9} we obtain
\begin{equation}
 \label{eqHJBOpsingleik11}
\begin{aligned}
& g^1(kS)\de (\tilde q^k_-+q^i_+ +\de) =\tilde q^k_- [ q^i_+ w^i_S + (q^i_+ +\de) w^k_I + (\tilde q^i_- -\tilde q^k_- -q^i_+ -\de)w^i_I ] \\
& \quad \quad \quad +[q^i_+(q^k_+ -\tilde q^k_- -q^i_+)+\de (q^k_+ -q^i_+)]w^k_S, \\
& g^1(iI)\de (\tilde q^k_-+q^i_+ +\de) =q^i_+ [ \tilde q^k_- w^k_I + (\tilde q^k_- +\de) w^i_S +(q^k_+ -q^i_+ -\tilde q^k_- -\de)w^k_S ] \\
& \quad \quad \quad +[\tilde q^k_-(\tilde q^i_--q^i_+ -\tilde q^k_-)+\de (\tilde q^i_- -\tilde q^k_-)]w^i_I, \\
\end{aligned}
\end{equation}

We can now check the conditions \eqref{eqHJBOpsingleik10}. Remarkably enough the r.h.s and l.h.s. of both inequalities always coincide for $\de=0$,
so that the actual condition arises from comparing higher terms in $\de$. In the first order with respect to the expansion in small $\de$
conditions \eqref{eqHJBOpsingleik10} turn out to take the following simple form
\begin{equation}
 \label{eqHJBOpsingleik12}
\tilde q^k_- (w^k_I-w^i_I) +w^k_S(q^k_+-q^i_+) \ge 0, \quad q^i_+ (w^i_S-w^k_S) +w^i_I(\tilde q^i_- -\tilde q^k_-) \ge 0.
\end{equation}

 From the last two equations of \eqref{eqHJBOpsingleik3} we can find $g(jS)$ and $g(jI)$ for $j\neq i,k$ yielding
\begin{equation}
 \label{eqHJBOpsingleik13}
\begin{aligned}
& g(jI) =g(iI) +\frac{1}{\la} [w^j_I-\de g(iI)+q^j_+ (g(iI)-g(kS))] +O(\la ^{-2}),   \\
& g(jS) =g(kS) +\frac{1}{\la} [w^j_S-\de g(kS)+\tilde q^j_- (g(iI)-g(kS))] +O(\la ^{-2}).
\end{aligned}
\end{equation}
From these equations we can derive the rest of the conditions \eqref{eqHJBOpsingleik2}, namely that $g(iI) \le g(jI)$ for $j\neq k$ and $g(kS) \le g(jS)$ for $j\neq i$.
In the first order in the small $\de$ expansion they become
\begin{equation}
 \label{eqHJBOpsingleik14}
q^j_+(w^i_I-w^k_S)+w^j_I (\tilde q^k_-+q^i_+) \ge 0, \quad \tilde q^j_-(w^i_I-w^k_S)+w^j_S (\tilde q^k_-+q^i_+) \ge 0.
\end{equation}

Since for small $\be_{ij}$, the difference $\tilde q^j_-  -q^j_-$ is small, we proved the following result.

\begin{prop}
\label{propstrikop}
Assume
\begin{equation}
 \label{eq1propstrikop}
\begin{aligned}
& q^j_+(w^i_I-w^k_S)+w^j_I (q^k_-+q^i_+) > 0, \quad j\neq k,  \\
& q^j_-(w^i_I-w^k_S)+w^j_S (q^k_-+q^i_+) > 0, \quad j\neq i, \\
& q^k_- (w^k_I-w^i_I) +w^k_S(q^k_+-q^i_+) > 0, \quad q^i_+ (w^i_S-w^k_S) +w^i_I(q^i_- -q^k_-) > 0.
\end{aligned}
\end{equation}
Then for sufficiently large $\la$, small $\de$ and small $\be_{ij}$  there exists a unique solution
to the stationary MFG consistency problem  \eqref{eqmainkineticbotnetmultist} and
\eqref{eqHJBbotnetmultidis}  with the optimal control $\hat u^{i,k}$, the stationary distribution is concentrated on strategies $i$ and $k$
with $x_{iI}^*$ being given by \eqref{eqkineticstsingleik5} or \eqref{eqkineticstsingleik6} up to terms of order $O(\la^{-1})$, and
it is stable; the optimal payoffs are given by \eqref{eqHJBOpsingleik8}, \eqref{eqHJBOpsingleik11}, \eqref{eqHJBOpsingleik13}.

Conversely, if for all sufficiently large $\la$ and small $\de$ there exists a solution
to the stationary MFG consistency problem  \eqref{eqmainkineticbotnetmultist} and
\eqref{eqHJBbotnetmultidis}  with the optimal control $\hat u^{i,k}$, then
\eqref{eqHJBOpsingleik12} and \eqref{eqHJBOpsingleik14} hold.
\end{prop}

\section{Main result}
\label{sectimedepsol}

By the general result already mentioned above, see \cite{BasRa14},
a solution of MFG consistency problem constructed above and considered on a finite time horizon
will define an $\ep$-Nash equilibrium for the corresponding game of finite number of players.
However, solutions given by Propositions \ref{propstriiop} and \ref{propstrikop} work only when the initial distribution
and terminal payoff are exactly those given by the stationary solution. Of course, it is natural to ask what happens for other initial
conditions. Stability results of  Propositions \ref{propstriiop} and \ref{propstrikop} represent only a step in the right direction here,
as they ensure stability only under the assumption that all (or almost all) players use from the very beginning
the corresponding stationary control, which might not be the case. To analyse the stability properly, we have to consider
the full time-dependent problem. For possibly time varying evolution $x(t)$ of the distribution, the time-dependent HJB equation
for the discounted optimal payoff $e^{-t\de}g$ of an individual player with any time horizon $T$ has form \eqref{eqHJBbotnetmultievdisc}.

In order to have a solution with a stationary $u$ we have to show that solving the linear equation obtained from
\eqref{eqHJBbotnetmultievdisc} by fixing this control will be consistent in the sense that this control
will actually give minimum in \eqref{eqHJBbotnetmultievdisc} in all times.

For definiteness, let us concentrate on the stationary control $\hat u^i$, the corresponding linear equation getting the form
\begin{equation}
\label{eqHJBOpsingleidisc1}
\left\{
\begin{aligned}
& \dot g(iI) + q^i_+(g(iS)-g(iI)) +w^i_I =\de g(iI), \\
& \dot g(iS) + q^i_-(g(iI)-g(iS)) + \sum_{k}\be_{ki} x_{kI}(t)(g(iI)-g(iS))+w_S^i =\de g(iS), \\
& \dot g(jI) + \la (g(iI)-g(jI)) +q^j_+(g(jS)-g(jI)) +w_I^j =\de g(jI), \quad j\neq i, \\
& \dot g(jS) + \la (g(iS)-g(jS)) +q^j_- (g(jI)-g(jS)) \\
& \quad \quad \quad + \sum_k\be_{kj} x_{kI}(t)(g(jI)-g(jS))+w_S^j =\de g(jS), \quad j\neq i,
\end{aligned}
\right.
\end{equation}
with the consistency requirement \eqref{eqHJBOpsinglei2}, but which has to hold now for time-dependent solution $g$.

\begin{theorem}
\label{thmainsingle}
Assume the strengthened form of \eqref{eqHJBOpsinglei11} holds, that is

\begin{equation}
 \label{eq1thmainsingle}
\frac{w^j_I-w^i_I}{w^i_I-w^i_S}  > \frac{q^j_+ -q^i_+}{q_-^i+q_+^i +\de}, \quad
\frac{w^j_S-w^i_S}{w^i_I-w^i_S}  > \frac{q^i_- -q^j_-}{q_-^i+q_+^i +\de}
\end{equation}
for all $j\neq i$. Assume moreover that
 \begin{equation}
 \label{eq2thmainsingle}
q^j_+ >q^i_+, \quad q^i_- > q^j_-
\end{equation}
for all $j\neq i$.
Then for any $\la>0$ and all sufficiently small $\be_{ij}$ the following holds.
For any $T >t$, any initial distribution $x(t)$  and any terminal values $g_T$ such that $g_T(jI)-g_T(jS) \ge 0$
for all $j$,  $g_T(iI)-g_T(iS)$ is sufficiently small and
 \begin{equation}
 \label{eq3thmainsingle}
g_T(iI) \le g_T(jI) \quad \text{and} \quad  g_T(iS) \le g_T(jS), \quad j\neq i,
\end{equation}
there exists a unique solution to the discounted MFG consistency equation such that $u$ is stationary and equals $\hat u^i$ everywhere.
Moreover, this solution is such that, for large $T-t$, $x(s)$ tends to the fixed point of Proposition  \ref{propstriiop} for $s\to T$ and
$g_s$ stays near the stationary solution of Proposition  \ref{propstriiop}  almost all time apart from a small initial period around $t$ and some
final period around $T$.
\end{theorem}

\begin{remark}
(i) The last property of our solution can be expressed by saying that the stationary solution provides the so-called turnpike
for the time-dependent solution, see e.g. \cite{KoYaTurnp} and \cite{ZasTurnp} for for reviews in stochastic and deterministic settings.
(ii) Condition  \eqref{eq2thmainsingle} is strong and can possibly be dispensed with by a more detailed analysis.
(iii) Similar time-dependent class of turnpike solutions can be constructed from the stationary control of Proposition \ref{propstrikop}.
\end{remark}

\begin{proof}
To show that starting with the terminal condition belonging to the cone specified by \eqref{eq3thmainsingle} we shall stay in this
 cone for all $t\le T$, it is sufficient to prove that on a boundary point of this cone that can be achieved by the evolution
 the inverted tangent vector of system \eqref{eqHJBOpsingleidisc1} is not directed outside of the cone. This (more or less obvious)
 observation is a performance of the general result of Bony, see e. g. \cite{redheffer}.
  From \eqref{eqHJBOpsingleidisc1} we find that
  \[
   \dot g(jI)-\dot g(iI) =(\la +\de) (g(jI)-g(iI)) +q^j_+(g(jI)-g(jS))-q^i_+(g(iI)-g(iS))  -(w_I^j-w^i_I).
   \]
   Therefore, the condition for staying inside the cone \eqref{eq3thmainsingle} for a boundary point with $g(jI)=g(iI)$ reads out as
 \begin{equation}
 \label{eq4thmainsingle}
(w_I^j-w^i_I) \ge q^j_+(g(jI)-g(jS))-q^i_+(g(iI)-g(iS)).
\end{equation}
Since
\[
0 \le g(jI)-g(jS) \le g(iI)-g(iS),
\]
a simpler sufficient condition for \eqref{eq4thmainsingle} is
 \begin{equation}
 \label{eq5thmainsingle}
(w_I^j-w^i_I) \ge (q^j_+ -q^i_+)(g(iI)-g(iS)).
\end{equation}

Subtracting the first two equations of \eqref{eqHJBOpsingleidisc1} we find that
\[
\dot g(iI) -\dot g(iS) =a(s) (g(iI)-g(iS)) -( w^i_I -w^i_S)
\]
with
\[
a(t) =q^i_+ +q^i_- +\de + \sum_{k}\be_{ki} x_{kI}(t).
\]
Consequently,
\begin{equation}
 \label{eq6thmainsingle}
g_t(iI) -g_t(iS)= \exp \{ -\int_t^T a(s) \, ds\}(g_T(iI) -g_T(iS)) +( w^i_I -w^i_S)\int_t^T  \exp \{ -\int_t^s a(\tau) \, d\tau\} ds.
\end{equation}
Therefore, as we assumed \eqref{eq2thmainsingle}, condition \eqref{eq5thmainsingle} will be fulfilled for all sufficiently small $g_T(iI) -g_T(iS)$ whenever
 \begin{equation}
 \label{eq7thmainsingle}
(w_I^j-w^i_I) > (q^j_+ -q^i_+)( w^i_I -w^i_S)\int_t^T  \exp \{ -\int_t^s a(\tau) \, d\tau\} ds.
\end{equation}
But since $a(t) \ge q^i_+ +q^i_- +\de$, we have
\[
\exp \{ -\int_t^s a(\tau) \, d\tau\}\le \exp \{-(s-t)(q^i_+ +q^i_- +\de)\},
\]
so that \eqref{eq5thmainsingle} holds if
 \begin{equation}
 \label{eq8thmainsingle}
\frac{w_I^j-w^i_I}{w^i_I -w^i_S} \ge \frac{q^j_+ -q^i_+}{q^i_+ +q^i_- +\de} \left(1- \exp \{ -(T-t)(q^i_+ +q^i_- +\de)\}\right),
\end{equation}
which is true under the first assumptions of \eqref{eq1thmainsingle} and \eqref{eq2thmainsingle}.

Similarly, to study a boundary point with $g(jS)=g(iS)$ we find that
 \[
   \dot g(jS)-\dot g(iS) =(\la +\de) (g(jS)-g(iS)) -(q^j_- + \sum_k\be_{kj} x_{kI})(g(jI)-g(jS))
   \]
   \[
    +(q^i_- +\sum_{k}\be_{ki} x_{kI} ) (g(iI)-g(iS))  -(w_S^j-w^i_S).
   \]
   Therefore, the condition for staying inside the cone \eqref{eq3thmainsingle} for a boundary point with $g(jS)=g(iS)$ reads out as
 \begin{equation}
 \label{eq9thmainsingle}
(w_S^j-w^i_S) \ge (q^i_- +\sum_{k}\be_{ki} x_{kI} ) (g(iI)-g(iS)) -(q^j_- + \sum_k\be_{kj} x_{kI})(g(jI)-g(jS)).
\end{equation}
 Now   $0 \le g(iI)-g(iS) \le g(jI)-g(jS)$, so that \eqref{eq9thmainsingle} is fulfilled if
  \begin{equation}
 \label{eq10thmainsingle}
(w_S^j-w^i_S) \ge (q^i_- +\sum_{k}\be_{ki} x_{kI} -q^j_- - \sum_k\be_{kj} x_{kI})(g(iI)-g(iS))
\end{equation}
for all times. Taking into account the requirement that all $\be_{ij}$ are sufficiently small, we find as above that
it holds under the second assumptions of \eqref{eq1thmainsingle} and \eqref{eq2thmainsingle}.

The last statement of the theorem concerning  $x(s)$ follows  from the observation that the eigenvalues of the linearized evolution
$x(s)$ are negative and well separated from zero implying the global stability of the fixed point of the evolution
 for sufficiently small $\be$. The last statement of the theorem concerning  $g(s)$ follows by similar stability argument
 for the linear evolution \eqref{eqHJBOpsingleidisc1} taking into account that away from the initial point $t$,
 the trajectory $x(t)$ stays arbitrary close to its fixed point.

\end{proof}

\end{document}